\documentclass[10pt,a4paper]{amsart}
\usepackage{a4}
\usepackage{amsfonts, amssymb}
\usepackage[all]{xy}
\usepackage{tikz}
\usetikzlibrary{shapes.geometric,positioning}

\def\b{{\mathbf b}}

\newtheorem{thm}{Theorem}
\newtheorem{cor}{Corollary}

\newtheorem{lem}{Lemma}

\newtheorem{prop}{Proposition}
\newcommand{\E}{\ensuremath{\mathbb E}}

\newcommand{\NN}{\mathbb{N}}

\newcommand{\Prob}{\mathbb P}
\newcommand{\Var}{\mathbb V}
\newcommand{\II}{\mathbb I}

\begin{document}
\title[Additive properties of sequences of pseudo $s$-th powers]{Additive properties\\ of sequences of pseudo $s$-th powers}

\begin{abstract}
In this paper, we study (random) sequences of pseudo $s$-th powers, as introduced by Erd\H os and R\'{e}nyi in 1960.
In 1975, Goguel proved that such a sequence is almost surely not an asymptotic basis of order $s$.
Our first result asserts that it is however almost surely a basis of order $s+\epsilon$ for any $\epsilon>0$.
We then study the $s$-fold sumset $sA=A+\cdots +A$ ($s$ times) and in particular the minimal size of an additive complement,
that is a set $B$ such that $sA+B$ contains all large enough integers. With respect to this problem, we prove quite precise
theorems which are tantamount to asserting that a threshold phenomenon occurs.
\end{abstract}


\author[J. Cilleruelo, J.-M. Deshouillers, V. Lambert, A. Plagne]{Javier Cilleruelo, Jean-Marc Deshouillers,\\ Victor Lambert, Alain Plagne}

\address{J.C.\\ Instituto de Ciencias Matem\'aticas (CSIC-UAM-UC3M-UCM) and
Departamento de Matem\'aticas\\
Universidad Aut\'onoma de Madrid\\
28049, Madrid, Espa\~na}
\address{J.M.D.\\ Institut Math\'{e}matique de Bordeaux\\ Bordeaux INP\\33405 Talence\\ France}
\address{V.L and A.P.\\ Centre de math\' ematiques Laurent Schwartz\\ \'Ecole polytechnique\\ 91128 Palaiseau Cedex\\ France}

\email{franciscojavier.cilleruelo@uam.es}
\email{jean-marc.deshouillers@math.u-bordeaux.fr}
\email{victor.lambert@math.polytechnique.fr}
\email{plagne@math.polytechnique.fr}

\thanks{The first author was supported by grants MTM 2011-22851 of MICINN  and ICMAT Severo Ochoa project SEV-2011-0087.
The second, third and fourth author were supported by an ANR grant C\ae sar, number ANR 12 - BS01 - 0011.
All the authors are thankful to \'{E}cole Polytechnique for making possible their collaboration}

\maketitle

\section{Introduction}
In their seminal paper of 1960, Erd\H os and R\'{e}nyi \cite{ER60} proposed a probabilistic model for sequences
$A$ growing like the $s$-th powers.
Explicitly, they built a probability space $(\mathcal{U}, \mathcal{T}, \Prob )$ and a sequence of independent Bernoulli
random variables $(\xi_n)_{n \in \mathbb{N}}$ with values in $\{0, 1\}$ such that
$$
\Prob (\xi_n =1)= \frac 1s n^{-1+1/s}\quad \text{and}\quad \Prob(\xi_n =0)= 1 -\frac1s  n^{-1+1/s}.
$$
To any $u \in \mathcal{U}$, they associate the sequence of integers $A = A_u$ such that $n \in A_u$ if and only if
$\xi_n(u) = 1$. In other words, the events $\{n \in A\}$ are independent and the probability that $n$ is in $A$ is equal
to $\Prob (n \in A) = n^{-1+1/s}/s$. The counting function of these random sequences $A$  satisfies almost surely
the asymptotic relation $|A\cap [1, x]|\sim x^{1/s}$ as $x$ tends to infinity \cite{ER60} (see also \cite{L95}), whence
the terminology {\em pseudo $s$-th powers}.

In 1975, Goguel \cite{G75} proved that, almost surely, the $s$-fold sumset
$$
sA = \{ a_1+\cdots+a_s \text{ with } \ a_i\in A\}
$$
has density $1-e^{-\lambda_s}$ where
$$
\lambda_s=\frac{\Gamma^s(1/s)}{s^s\ s!}
$$
(a quantity appearing everywhere in the present study) and thus, almost surely, that $A$ is not an asymptotic basis of order $s$
(from now on, the word `asymptotic' will be omitted since there is no ambiguity).
Indeed it has been proved recently (see \cite{CD})  that the sequence $(b_n)_{n \in \mathbb{N}}$ of ordered elements in $sA$
has, almost surely, infinitely many gaps of logarithm size that is,
\begin{equation}\label{limsup}
\limsup_{n\to +\infty} \frac{b_{n+1}-b_n}{\log b_n}=\frac{1}{\lambda_s}.
\end{equation}

In contrast to the result of Goguel quoted above, Deshouillers and Iosifescu,
as a by-product of their study on the probability that an integer is not a sum of $s+1$ $s$-th powers,
proved in \cite{DI}, however,
that almost surely a sequence of pseudo $s$-th powers is a basis of order $s+1$. Here we will
make more precise this threshold-type phenomenon by using the concept of a basis of order $s+\epsilon$
introduced in \cite{Ci1}: We say that $A$ is a {\em basis of order $s+\epsilon$} if any large enough positive
integer $n$ can be written in the form
$$
n=a_1+\cdots +a_{s+1},\qquad \text{ with } a_i\in A,\qquad a_{s+1}\le n^{\epsilon}.
$$

Our first result is the fact that almost surely a sequence of pseudo $s$-th powers is a basis of order
$s+\epsilon$ for any $\epsilon>0$. Indeed we prove this result in the following stronger form.

\begin{thm}
\label{th1}
Let  $s\ge 2$ be an integer and $c > \left(\lambda_s (1- 2 \lambda_s) \right)^{-1}$.  Almost surely,
a sequence of pseudo $s$-powers $A$ has the following property:  any large enough integer $n$
can be written in the form
$$
n=a_1+\cdots+a_{s+1},\qquad  \text{ with } a_i\in A,\qquad a_{s+1}< (c\log n)^s.
$$
\end{thm}

We have some reason to believe that the above statement is no longer valid if $c < \lambda_s^{-1}$; 
this point will be discussed at the end of Section \ref{sec3}.  Simply notice now that $\lambda_s < 1/2$ for $s\ge 2$.

A second aim of the paper is the study of how fast an additive complement sequence of $sA$ must grow.
We first prove the following theorem.

\begin{thm}
\label{th2}
Let $s$ be an integer $s \geq 2$.
Let $B$ be a fixed sequence satisfying $$\liminf_{n\to \infty}\frac{B(n)}{\log n}> \lambda_s^{-1}.$$
Then a sequence of pseudo $s$-powers $A$ has, almost surely, the following property:
any large enough integer $n$ can be written in the form
$$
n=a_1+\cdots+a_s+b,\qquad  \text{ with distinct } a_i\in A \text{ and with } b\in B.
$$
\end{thm}

We then prove that Theorem \ref{th2} is sharp in the sense that the constant $\lambda_s^{-1}$
intervening in this result cannot be substituted by a smaller constant.

\begin{thm}
\label{th2bis}
Let $s$ be an integer $s \geq 2$.
Let $B$ be a fixed sequence satisfying $$\liminf_{n\to \infty}\frac{B(n)}{\log n}<\lambda_s^{-1}.$$
Then a sequence of pseudo $s$-powers $A$ has, almost surely, the following property:
there are infinitely many integers $n$ that cannot be written in the form
$$
n=a_1+\cdots+a_s+b,\qquad  \text{ with distinct } a_i\in A  \text{ and with }  b\in B.
$$
\end{thm}

In view of Theorems \ref{th2} and \ref{th2bis} it is a natural question to ask for the behaviour of those sequences $B$ with
\begin{equation}
\label{BBB}
\liminf_{n\to \infty}\frac{B(n)}{\log n}=\lambda_s^{-1}.
\end{equation}
In our final section, we will show that there are sequences satisfying \eqref{BBB} and the conclusion of Theorem \ref{th2}
while there are other sequences that satisfy \eqref{BBB} and the conclusion of Theorem \ref{th2bis}.

The paper is organized as follows. Section \ref{sec2} is composed of several lemmas which
will be useful in the proofs of the three theorems. Section \ref{sec3} contains the proof of
Theorem \ref{th1} which will be presented with precise estimates. Section \ref{sec4} contains
the proof of Theorem \ref{th2} and Section \ref{sec5} that of Theorem \ref{th2bis}. Finally, Section
\ref{sec6} contains the discussion about sequences at the threshold (that is, satisfying \eqref{BBB}).
In order to avoid overcomplicated and lengthy formulations, these proofs,
which rely on analogous but simpler principles and computation, will be written down with slightly less details.


\section{Preparatory lemmas and prerequisite}
\label{sec2}

For our purpose, we shall need a few elementary or more or less classical results.
The first one is technical and we shall use the standard Vinogradov $\ll$ notation for ``less than a constant time'';
in the present paper the constants will always depend on the
parameter $s \geq 2$, but only on it. We will not recall this dependency in the $\ll$ notation.

\begin{lem}
\label{t1}
Let $s$ and $t$ be two integers such that $s\geq 2$ and $1\le t\le s-1$.  We have
\begin{itemize}
\item [(i)] for $z \geq 1$,
$$
\sum_{\substack{1 \leq x_1,\dots,x_t \\ x_1+\cdots +x_t=z}}(x_1\cdots x_t)^{-1+1/s}\ll z^{-1+t/s},
$$
\item [(ii)]  for $z \geq 2$,
$$
\sum_{\substack{1 \leq x_1,\dots,x_t \\ x_1+\cdots +x_t<z}}(x_1\cdots x_t)^{-1+1/s}\big (z-(x_1+\cdots +x_t)\big )^{-2t/s}\ll z^{-1/s}\log z,
$$
\item [(iii)] if $g$ is a positive function satisfying $g(z)=o(z)$ as $z$ tends to infinity, then
$$
\lim_{\substack{ z \to + \infty\\ z \in \NN}}\sum_{\substack{g(z)\le x_s<\cdots <x_1\\ x_1+\cdots +x_s=z}}(x_1\cdots x_s)^{-1+1/s} = s^s\lambda_s.
$$
\end{itemize}
\end{lem}

\begin{proof}
Points (i) and (ii) in this lemma appear as Lemma 1 of \cite{CD}, taking $a_1=\cdots=a_s=1$.
The special case $g(z)=1$ of (iii) appears there also.
To extend it to our setting, it will be enough to prove that
$$
\sum_{\substack{1\le x_s \le g(z)\\ 1 \leq x_{s-1}<\cdots <x_1 \\ x_1+\cdots +x_s=z}}(x_1\cdots x_s)^{-1+1/s}=o(1).
$$
To see this we use (i) with $t=s-1$ and bound this sum as
\begin{eqnarray*}
\sum_{\substack{1\le x_s\le g(z)\\ 1 \leq x_{s-1}<\cdots <x_1 \\ x_1+\cdots +x_s=z}}\hspace{-.5cm}(x_1\cdots x_s)^{-1+1/s}
& \leq &\sum_{1\le x_s <g(z)} x_s^{-1+1/s}\hspace{-.5cm}
\sum_{\substack{1 \leq x_{s-1}<\cdots <x_1\\x_1+\cdots +x_{s-1}=z-x_s}}\hspace{-.5cm}(x_1\cdots x_{s-1})^{-1+1/s}\\
&\ll &\sum_{1\le x_s<g(z)}x_s^{-1+1/s}(z-x_s)^{-1/s}\\
&\ll & (z-g(z))^{-1/s} \sum_{1\le x_s<g(z)}x_s^{-1+1/s}\\
&\ll & \left( \frac{g(z)}{z} \right)^{1/s}\\  
&=&o(1),
\end{eqnarray*}
as needed.
\end{proof}

Here are now a few more or less classical tools from probability theory. The first basic tool is
Chebychev's inequality in the following form, suitable for our purpose:
\begin{equation}
\label{cheby}
\Prob \left(X<\frac{\mathbb{E}[X]}{2}\right ) \le \frac{4\Var(X)}{\mathbb{E}[X]^2} .
\end{equation}
Here and everywhere in this paper, the symbols $\Prob, \E$ and $\Var$ denote
respectively the probability, the mathematical expectation and the variance.

The Borel-Cantelli Lemma is another basic and well known tool in probability (see for instance Lemma 8.6.1 in \cite{AS}).
We recall it here for the sake of completeness.

\begin{thm}[Borel-Cantelli Lemma]
\label{Borel}
Let $(F_i)_{i \in \mathbb{N}}$ be a sequence of events. 
If $\sum_{i=1}^{+ \infty} \Prob (F_i) < + \infty$ then,

with probability $1$, only finitely many of the events $F_i$ occur.

\end{thm}

Next, we will need 
two correlation inequalities due to Janson \cite{JLR} (see also \cite{BS89})
which are known as
``Janson's correlation inequalities". Up to the ordering of the
elements, this is Theorem 8.1.1 in \cite{AS}.

We shall use the following notation : if $\Omega$ is a set, then for any two subsets $\omega, \omega' $ of $\Omega$, the notation $\omega \sim \omega'$  means 		
that $\omega\ne \omega'$ and $\omega \cap \omega' \neq \emptyset$. Moreover, we use the standard notation $E^c$ for the complementary event of an event $E$.

\begin{thm}[Janson's inequalities]
\label{des1}
Let $(E_{\omega})_{\omega\in \Omega}$ be a finite collection of events indexed by subsets of a set $\Omega$
and assume that $P(E_{\omega})\le 1/2$ for any $\omega\in \Omega$.  Then the quantity $\Prob \big (\bigcap_{\omega\in \Omega}E_{\omega}^c\big )$
satisfies
\begin{itemize}
\item[(i)]  the lower bound
$$
\Prob \big (\bigcap_{\omega\in \Omega}E_{\omega}^c\big ) \ge \prod_{\omega\in \Omega}\Prob (E_{\omega}^c)
$$
and
\item[(ii)] the upper bound
$$
\Prob \big (\bigcap_{\omega\in \Omega}E_{\omega}^c\big )   \le
\Big( \prod_{\omega\in \Omega}\Prob (E_{\omega}^c) \Big) \exp \Big( 2\sum_{\substack{\omega, \omega' \in \Omega\\ \omega \sim \omega'}} \Prob(E_{\omega} \cap E_{\omega'})\Big).
$$
\end{itemize}
\end{thm}


\section{Proof of Theorem \ref{th1}}
\label{sec3}

Let $c > \left(\lambda_s (1- 2 \lambda_s) \right)^{-1}$, as in the statement of Theorem \ref{th1}; we recall that $\lambda_s < 1/2$ when $s\ge 2$.

We represent the sets of $s+1$ distinct elements in the form $\omega=\{ x_1,\dots,x_{s+1}\}$ with
$$
x_{s+1} <\cdots < x_1.
$$
We also denote $\sigma(\omega)=x_1+\cdots+x_{s+1}$ and, for each $n$, we let
$$
\Omega_n=\{\omega \text{ such that } \sigma(\omega)=n,\   x_{s+1}<(c\log n)^s \text{ and }    (c\log n)^s < x_s    \}.
$$

If we denote by $E_{\omega}$ the event $\omega\subset A$ and denote $\II$ the indicator 
function of an event, the function
$$
r(n,A)=\sum_{\omega\in \Omega_n}\II (E_{\omega})
$$
counts the number of representations of $n$ of the form $n=x_1+\cdots+x_{s+1}$, where
\begin{eqnarray*}
x_i\in A,\quad (c\log n)^s<x_{s} <\cdots < x_1,\quad \text{ and }\quad  x_{s+1}<(c\log n)^s.
\end{eqnarray*}

By the Borel-Cantelli Lemma, Theorem \ref{th1} will be proved as soon as we prove that the series $\Prob(r(n,A)=0)$ converges. We follow the strategy introduced in \cite{DI}. Using our definition and Janson's second correlation inequality, we have
\begin{eqnarray}\notag
\Prob(r(n,A)=0) = \Prob\big(\bigcap_{\omega \in \Omega_n} E_{\omega}^c\big) \le \prod_{\omega \in \Omega_n} \Prob\big(E_{\omega}^c\big)\times \exp\left(2\Delta_n\right),
\end{eqnarray}
with
\begin{eqnarray}\label{Delta}
 \Delta_n=\sum_{\substack{\omega, \omega'\in \Omega_n\\ \omega \sim \omega'}}\mathbb P(E_{\omega}\cap E_{\omega'}).
\end{eqnarray}

We first study the product.

\begin{lem}\label{prod}
When $n$ tends to infinity, we have
$$
\prod_{\omega \in \Omega_n} \Prob\big(E_{\omega}^c\big) = \exp\big(-(1+o(1))c\lambda_s\log n\big).
$$
\end{lem}

\begin{proof}
We compute
\begin{eqnarray*}
\sum_{\omega\in \Omega_n}\Prob (E_{\omega}) = 
\frac 1{s^{s+1}}\sum_{1 \leq x_{s+1}<(c\log n)^s}x_{s+1}^{1/s-1}
\sum_{\substack{(c\log n)^s< x_s < \cdots < x_1 \\ x_1+\cdots +x_s=n-x_{s+1}}}(x_1\dots x_s)^{1/s-1}.
\end{eqnarray*}
For each $x_{s+1}<(c\log n)^s$, we may apply Lemma \ref{t1} (iii) with $z=n-x_{s+1} \sim n$ which gives
$$
\sum_{\omega\in \Omega_n}\Prob (E_{\omega}) =(1+o(1))\frac{\lambda_s}s\sum_{1 \leq x_{s+1}<(c\log n)^s  }x^{1/s-1}=(1+o(1))c\lambda_s\log n,
$$
and the result follows from this and the simple relation
\begin{equation}\notag
\prod_{\omega \in \Omega_n} \Prob\big(E_{\omega}^c\big) = \exp\left(\sum_{\omega \in \Omega_n} \log(1-(\Prob (E_{\omega}))\right) = \exp\left(-(1+o(1))\sum_{\omega\in \Omega_n}\Prob (E_{\omega})  \right)
\end{equation}
\end{proof}

We now come to the correlation term $\Delta_n$ defined in (\ref{Delta}).

\begin{lem}
\label{delta}
When $n$ tends to infinity, one has
$$
\Delta_n\le (1+o(1))c \lambda_s^2\log n.			
$$
\end{lem}

\begin{proof}
In order to decompose the sum defining $\Delta_n$, we introduce
$$
\Delta_n(k)=\sum_{\substack{\omega, \omega'\in \Omega_n\\\omega\sim \omega'\in \Omega_n\\x_{s+1}=y_{s+1}\\ |\omega\cap \omega'|=k}}
\mathbb P(E_{\omega}\cap E_{\omega'})
\quad \text{ and }\quad
\Delta_n'(k)=\sum_{\substack{\omega, \omega'\in \Omega_n\\\omega\sim \omega'\in \Omega_n\\x_{s+1}\ne y_{s+1}\\ |\omega\cap \omega'|=k}}
\mathbb P(E_{\omega}\cap E_{\omega'})
$$
so that
$$
\Delta_n=\sum_{k=1}^{s-1}\Delta_n(k)+\sum_{k=1}^{s-1}\Delta'_n(k).
$$
We study each term of this formula separately and shall observe that the main contribution comes from $\Delta_n(1)$.

(i) We compute that
\begin{eqnarray*}
\Delta_n(1) & = &\frac{1}{s^{2s+1}} \sum_{\substack{(c\log n)^s < x_s <\cdots < x_1\\  (c\log n)^s < y_s <\cdots <y_1\\ x_{s+1}<(c\log n)^s\\
x_{1}+\cdots +x_{s}=y_{1}+\cdots +y_s=n-x_{s+1}\\ x_i \neq y_j \text{ for any indices } i \text{ and } j}}
(x_1\cdots x_{s+1}y_{1}\cdots y_s)^{-1+1/s} \\
		& \leq & \frac{1}{s^{2s+1}} \sum_{1 \leq x_{s+1}<(c \log n)^s}x_{s+1}^{-1+1/s} \left(\sum_{\substack{1 \leq x_s <\cdots < x_1\\
x_{1}+\cdots +x_{s}=n-x_{s+1}}}(x_{1}\cdots x_s)^{-1+1/s}\right )^2.
\end{eqnarray*}

For each $x_{s+1}<(c\log n)^s$, we may apply Lemma \ref{t1} (iii) with $z=n-x_{s+1} \sim n$ which yields
\begin{eqnarray*}
\Delta_n(1) &\le &(1+o(1))    \frac{1}{s^{2s+1}} (s^{s}\lambda_s)^2 \sum_{1 \leq x_{s+1}<(c \log n)^s} x_{s+1}^{-1+1/s}\\
&\le &(1+o(1)) c \lambda_s^2 \log n
\end{eqnarray*}
as $n$ tends to infinity.

(ii) For $2 \le k \le s-1$, we have
\begin{eqnarray*}
\Delta_n(k) & = & \frac{1}{s^{2s+2-k}}  \sum_{\substack{K,K' \subset \{1,\dots,s \} \\  |K|=|K'|=k-1}} \hspace{-1cm}
\sum_{\substack{
(c\log n)^s < x_s <\cdots < x_1\\
(c\log n)^s < y_s <\cdots <y_1\\
1 \leq x_{s+1}<(c\log n)^s\\
\sum_{i \not\in K}x_i=\sum_{i\not\in K'} y_i= n-\left( \sum_{i \in K}x_i \right)-x_{s+1}\\
x_i \neq y_j \text{ for any indices } i \not\in K \text{ and } j \not\in K'\\
\{ x_i \text{ for } i \in K \} = \{ y_i \text{ for } i \in K' \} } } \hspace{-.5cm}
\left( \left (\prod_{i=1}^{s+1} x_i\right )  \left (\prod_{i \not\in K'} y_i\right ) \right)^{-1+1/s} \\
&\ll & \sum_{\substack{
(c\log n)^s < x_s <\cdots < x_1\\
(c\log n)^s < y_s <\cdots <y_k\\
1 \leq x_{s+1}<(c\log n)^s\\
x_{k}+\cdots +x_{s}=y_{k}+\cdots +y_s=n-(x_1+\cdots + x_{k-1} + x_{s+1})}}
\left( x_1\cdots x_{s+1}y_{k}\cdots y_s \right)^{-1+1/s},
\end{eqnarray*}
after regrouping together similar terms.
Thus,
\begin{eqnarray*}
\Delta_n(k)&\ll&	\sum_{1 \leq x_{s+1}<(c \log n)^s} x_{s+1}^{-1+1/s}
			\sum_{\substack{(c\log n)^s <  x_1,\dots,x_{k-1}\\ x_1+\cdots+x_{k-1}<n-x_{s+1}}}(x_1\cdots x_{k-1})^{-1+1/s} \\
& &\hspace{2.5cm}\times
\left (\sum_{\substack{(c\log n)^s <  x_k,\dots,x_{s}\\ x_{k}+\cdots +x_{s}=n-x_1-\cdots -x_{k-1}-x_{s+1}}}(x_{k}\cdots x_s)^{-1+1/s}\right )^2.
\end{eqnarray*}
We first use Lemma \ref{t1} (i) with $z= n-x_1-\cdots -x_{k-1}-x_{s+1} \geq 1$ to bound the last term. We obtain
\begin{eqnarray*}
\Delta_n(k) &\ll & \sum_{1 \leq x_{s+1}<(c\log n)^s}x_{s+1}^{-1+1/s}\sum_{\substack{1 \leq x_1,\dots,x_{k-1}\\x_1+\cdots+x_{k-1}<n-x_{s+1}}}(x_1\cdots x_{k-1})^{-1+1/s} \\
& &\hspace{4cm} \times \big (n-x_{s+1}-(x_1+\cdots +x_{k-1})\big )^{-2(k-1)/s}
\end{eqnarray*}
and we apply now Lemma \ref{t1} (ii) with $z= n-x_{s+1}\geq 2$ which gives
\begin{eqnarray*}
\Delta_n(k)&\ll & \sum_{1 \leq x_{s+1}<(c\log n)^s}x_{s+1}^{-1+1/s}(n-x_{s+1})^{-1/s}\log (n-x_{s+1})\\
&\ll & n^{-1/s}\log n \sum_{1 \leq x_{s+1}<(c\log n)^s}x_{s+1}^{-1+1/s}\\
&\ll &  n^{-1/s}\log^2 n.
\end{eqnarray*}

(iii) Finally, for $1 \le k \le s-1$, using a similar decomposition, we obtain
$$
\Delta'_n(k)\ll\sum_{\substack{
1 \leq x_s <\cdots < x_1\\
1 \leq y_s <\cdots <y_{k+1}\\
1 \le x_{s+1},y_{k+1}<(c \log n)^s\\
x_{k+1}+\cdots +x_{s+1}=y_{k+1}+\cdots +y_{s+1}=n-(x_1+\cdots+x_{k})}}
(x_1\cdots x_{s+1}y_{k+1}\cdots y_{s+1})^{-1+1/s}.
$$

Thus,
$$
\Delta'_n(k) \ll \sum_{\substack{1 \leq x_1,\dots,x_{k}\\x_1+\cdots+x_k<n}}(x_1\cdots x_{k})^{-1+1/s}S(n;x_1,\dots,x_k)^2
$$
where
$$
S(n;x_1,\dots,x_k)=\sum_{\substack{1 \leq x_{k+1},\dots,x_s\\ 1 \leq x_{s+1}<(c\log n)^s \\ x_{k+1}+\cdots +x_{s+1}=n-(x_1+\cdots +x_{k})}}(x_{k+1}\cdots x_{s+1})^{-1+1/s}.
$$
We now study this sum and distinguish two cases.

(a) First, if $x_1+\cdots +x_k<n-2(c\log n)^s$ then
$$
S(n;x_1,\dots,x_k)=\sum_{1 \leq x_{s+1}<(c\log n)^s} x_{s+1}^{-1+1/s}
\hspace{-1cm}
\sum_{\substack{
1 \leq x_{k+1},\dots,x_s\\
x_{k+1}+\cdots +x_{s}=n-x_{s+1}-(x_1+\cdots +x_{k})}}(x_{k+1}\cdots x_{s})^{-1+1/s}
$$
which can be bounded above, using Lemma \ref{t1} (i) for each internal sum with $z=n-x_{s+1}-(x_1+\cdots +x_{k}) \geq 1$ by
\begin{eqnarray*}
\ &\ll & \sum_{1 \leq x_{s+1}<(c\log n)^s}x_{s+1}^{-1+1/s}(n-(x_1+\cdots +x_{k})-x_{s+1})^{-k/s}\\
&\ll & (n-(x_1+\cdots +x_k))^{-k/s}\log n.
\end{eqnarray*}

(b) Second, in the case $n-2(c\log n)^s\le x_1+\cdots +x_k<n$, we have 
using Lemma \ref{t1} (i) with $z=n-(x_1+\cdots +x_{k})\geq 1$,
\begin{eqnarray*}
S(n;x_1,\dots,x_k)&\le &\sum_{\substack{1 \leq x_{k+1},\dots, x_{s+1}\\x_{k+1}+\cdots +x_{s+1}=n-(x_1+\cdots+x_k)}}(x_{k+1}\cdots x_{s+1})^{-1+1/s}
\\&\ll & (n-(x_1+\cdots+x_k))^{(1-k)/s}\\
&\ll& 1.
\end{eqnarray*}

From these bounds (a) and (b) on the sums $S(n;x_1,\dots,x_k)$ we derive
\begin{eqnarray*}
\Delta'_n(k)&\ll&\sum_{\substack{1 \leq x_1,\dots,x_{k}\\x_1+\cdots+x_k<n-2(c\log n)^s}}(x_1\cdots x_{k})^{-1+1/s}S(n;x_1,\dots,x_k)^2\\
&	& \hspace{2cm} + \hspace{-.5cm}\sum_{\substack{1 \leq x_1,\dots,x_{k}\\n-2(c\log n)^s\le x_1+\cdots+x_k<n}} \hspace{-.5cm} (x_1\cdots x_{k})^{-1+1/s}S(n;x_1,\dots,x_k)^2\\
&\ll 	& \log^2 n \hspace{-.5cm}\sum_{\substack{1 \leq x_1,\dots,x_{k}\\x_1+\cdots+x_k<n-2(c\log n)^s}}(x_1\cdots x_{k})^{-1+1/s}(n-(x_1+\cdots +x_k))^{-2k/s}\\
&	&\hspace{3cm} + \sum_{n-2(c\log n)^s \le r<n   } \hspace{.2cm} \sum_{\substack{1 \leq x_1,\dots,x_{k}\\ x_1+\cdots+x_k=r}}(x_1\cdots x_{k})^{-1+1/s}\\
&\ll 	& \log^2 n \hspace{-.5cm}\sum_{\substack{1 \leq x_1,\dots,x_{k}\\x_1+\cdots+x_k<n}}(x_1\cdots x_{k})^{-1+1/s}(n-(x_1+\cdots +x_k))^{-2k/s}\\
&	&\hspace{3cm} + \sum_{n-2(c\log n)^s \le r<n   } \hspace{.2cm} r^{-1+k/s}\\
&\ll   & n^{-1/s} \log^3 n +   n^{-1+k/s}(\log n)^s\\
&\ll  & n^{-1/s} \log^{s+1} n
\end{eqnarray*}
where we use Lemma \ref{t1} (ii) applied with $t=k$ and $z=n$ in the first term and Lemma \ref{t1} (i) with $t=k$ and $z=r$ for each internal term of the second sum.

The conclusion of the lemma follows from collecting the estimates of (i), (ii) and (iii) just obtained.
\end{proof}

Gathering the results of Lemma \ref{prod} and Lemma \ref{delta}, we obtain
\begin{equation}\notag
\Prob(r(n,A)=0) \le \exp\left(-(1+o(1))c \lambda_s (1 - 2 \lambda_s)\log n\right),
\end{equation}
which is the general term of a convergent series as soon as $c \lambda_s (1 - 2 \lambda_s) > 1$; this ends the proof of Theorem \ref{th1}. As was noticed in \cite{DI}, the factor 2 occurring in Janson's inequality may be reduced to any constant larger than $1$; however, the correlation term is still of the same order of magnitude as the main term.

What about a reverse result? Janson's first correlation inequality leads to 
\begin{equation}\label{lower}
\Prob(r(n,A)=0) \ge \exp\left(-(1+o(1))c \lambda_s\log n\right),
\end{equation}
which is the general term of divergent series as soon as $c \lambda_s < 1$. A first minor point is that $r(n,A)$ only counts special representations (pairwise distinct summands and only one which is less than $(c \log n)^s$) but it is not difficult to obtain a bound like (\ref{lower}) taking into account all the representations. More seriously, to apply the "reverse" Borel-Cantelli Lemma, some independence between the events $\{r(n,A)=0\}$ is required; unfortunately, we just miss the condition given in \cite{ER59}.


\section{Proof of Theorem \ref{th2}}
\label{sec4}

By assumption, there is some
$$
c>\lambda_s^{-1}
$$
such that the fixed sequence $B$ has a counting function satisfying
\begin{equation}\label{Bi}
B(n) \ge c(1+o(1))\log n.
\end{equation}
For each integer $n$, we define $m=m(n)$ to be the smallest positive integer such that
$$
B(m)=\left \lfloor \frac{c+\lambda_s^{-1}}2 \log n\right \rfloor .
$$
We observe, by \eqref{Bi} and the definition of $m$, that
$$
m=n^{(1+o(1))\frac{1+\lambda_s^{-1}/c}{2}} =o(n),
$$
which will be used through the proof.

We represent the sets of $s$ distinct elements in the form $\omega=\{x_1,\dots,x_{s}\}$ with $x_1>\cdots >x_{s}$. We also denote
 $\sigma(\omega)=x_1+\cdots+x_{s}$ and for each $n$ let
$$
\Omega_n=\{\omega \text{ such that } \sigma(\omega)=n-b \text{ for some } b\in B,\ b<m\},
$$
where

If we denote by $E_{\omega}$ the event $\omega\subset A$, then the event ``$n$ cannot be written in the form
$n=a_1+\cdots+a_s+b$ with $a_1>\dots >a_s,\ a_i\in A,\  b\in B,\ b<m$", which we denote by $F_n$, can be expressed in the form
$$
F_n=\bigcap_{\omega\in \Omega_n}E_{\omega}^c.
$$

We start with two lemmas.

\begin{lem}
\label{B}
One has
$$
\sum_{\omega\in \Omega_n}\Prob (E_{\omega}) = (1+o(1))\frac{c\lambda_s+1}2 \log n.
$$
\end{lem}

\begin{proof}
Indeed, using Lemma \ref{t1} (iii), we compute
\begin{eqnarray*}
\sum_{\omega\in \Omega_n}\Prob (E_{\omega})&=&\sum_{b<m} \frac{1}{s^s} \sum_{\substack{1 \leq x_1<\cdots <x_s\\x_1+\cdots +x_s=n-b}}(x_1\cdots x_s)^{-1+1/s}\\
 &=&(1+o(1)) B\left(m\right) \lambda_s\\
&= &(1+o(1))\frac{c\lambda_s+1}2\log n. \end{eqnarray*}
\end{proof}

\begin{lem}
\label{BB456}
One has
$$
\sum_{\substack{\omega\sim \omega'\\ \omega,\omega'\in \Omega_n}}\Prob (E_{\omega}\cap E_{\omega'})\ll n^{-1/s}(\log n)^3.
$$
\end{lem}

\begin{proof}
We can write
$$
\sum_{\substack{\omega\sim \omega'\\ \omega,\omega'\in \Omega_n}}\Prob (E_{\omega}\cap E_{\omega'})=
\sum_{\substack{b \le  b'<m\\ b,b'\in B}}\sum_{k=1}^{s-1}\Delta_n(k;b,b')
$$
where
$$
\Delta_n(k;b,b')=\sum_{\substack{\omega,\omega'\in \Omega_n\\\sigma(\omega)=n-b,\ \sigma(\omega')=n-b'\\|\omega\cap \omega'|=k}}\Prob (E_{\omega}\cap E_{\omega'}).
$$
Thus,
\begin{eqnarray*}
\Delta_n(k;b,b') &\ll & \hspace{-.3cm} \sum_{\substack{1 \leq x_1<\cdots <x_k\\x_1+\dots +x_k<n-b'}}
						\hspace{-.5cm} (x_1\cdots x_k)^{-1+1/s}
				\left( \sum_{\substack{x_{k+1},\dots ,x_s\\ x_{k+1}+\cdots +x_s=n-b-(x_1+\cdots +x_k)}}(x_{k+1}\cdots x_s)^{-1+1/s}\right ) \\
				& &\hspace{3.5cm} \times \left ( \sum_{\substack{y_{k+1},\dots,y_s\\y_{k+1}+\cdots +y_s=n-b'-(x_1+\cdots+x_k)}}(y_{k+1}\cdots y_s)^{-1+1/s}\right).
\end{eqnarray*}
But Lemma \ref{t1} (i) gives, for $\zeta =b$ or $b'$,
\begin{eqnarray*}
\sum_{\substack{1 \leq x_{k+1},\dots ,x_s\\x_{k+1}+\cdots +x_s=n-\zeta-(x_1+\cdots +x_k)}}(x_{k+1}\cdots x_s)^{-1+1/s}&\ll &(n-\zeta-(x_1+\cdots+x_k))^{-k/s}\\
&\ll & (n-b'-(x_1+\cdots+x_k))^{-k/s}
\end{eqnarray*}
and applying this bound and later  Lemma \ref{t1} (ii) we obtain
\begin{eqnarray*}
\Delta_n(k;b,b') &\ll & \hspace{-.3cm} \sum_{\substack{1 \leq x_1<\cdots <x_k\\x_1+\dots +x_k<n-b'}}
						\hspace{-.5cm} (x_1\cdots x_k)^{-1+1/s}(n-b'-(x_1+\cdots+x_k))^{-2k/s}\\
&\ll & (n-b')^{-1/s}\log(n-b').				
\end{eqnarray*}

Adding all the contributions, it follows
$$
\sum_{\substack{\omega\sim \omega'\\ \omega,\omega'\in \Omega_n}}
\Prob (E_{\omega}\cap E_{\omega'})\ll \sum_{\substack{b \le b'<m\\ b,b'\in B}}(n-b')^{-1/s}\log(n-b')\ll B(m)^2 n^{-1/s}\log n.
$$
and using $B(m) \ll \log n$ concludes the proof of the lemma.
\end{proof}

We now come to the very proof of the Theorem.
By Janson's second inequality (Theorem \ref{des1} (ii)) we obtain the following upper bound for $\Prob (F_n)$, namely
$$
\Prob (F_n)\le\prod_{\omega\in \Omega_n}\left (1-\Prob (E_{\omega})\right )
\exp\left (2\sum_{\substack{\omega\sim \omega'\\ \omega,\omega'\in \Omega_n}} \Prob (E_{\omega}\cap E_{\omega'})\right )
$$
which, using the inequality $\log(1-x)<-x$ (valid for $x>0$) yields
\begin{equation}
\label{log}
\log \Prob (F_n) \le -\sum_{\omega\in \Omega_n}\Prob (E_{\omega})+2\sum_{\substack{\omega\sim \omega'\\ \omega,\omega'\in \Omega_n}}\Prob (E_{\omega}\cap E_{\omega'}).
\end{equation}

Plugging in \eqref{log} the estimates obtained in Lemmas \ref{B} and \ref{BB456} we get
$$
\log \Prob (F_n)\le - (1+o(1))\frac{c\lambda_s+1}2\log n,
$$
so that
$$
\Prob (F_n)\le n^{-(1+o(1))\frac{c\lambda_s+1}2}.
$$
If $c>\lambda_s^{-1}$ then $(c\lambda_s+1)/2>1$ and the sum $\sum_n \Prob (F_n)$ is finite. The Borel-Cantelli Lemma implies
that, almost surely, only a finite number of events $F_n$ can occur and we are done.


\section{Proof of Theorem \ref{th2bis}}
\label{sec5}

We use the same kind of notation as in the proof of Theorem \ref{th2} but now
$$
\Omega_n=\{\omega \text{ such that }  \sigma(\omega)=n-b \text{ for some } b\in B\}.
$$

We define the event $F_n$: ``$n$ cannot be written in the form $n=x_1+\cdots +x_s+b$ with $x_1,\dots,x_s\in A$,
$x_s<\cdots <x_1$ and $b\in B$."
In other words,
$$
F_n=\bigcap_{\omega\in \Omega_n}E_{\omega}^c.
$$

The hypothesis of Theorem \ref{th2bis} is tantamount to writing
$$
\liminf_{n \to + \infty} \frac{B(n)}{\log n}=c
$$
for some $c<\lambda_s^{-1}.$ Then
there exists a sequence $\left(N_i\right)_{i\in\mathbb{N}}$ of integers such that
\begin{equation}\label{N}B(N_i)=c(1+o(1))\log N_i.\end{equation}

In all this proof, if $N$ is some integer, we shall say that a positive integer $n$ is {\em good (for N)} if
$N/2\le n\le N$ and
$$
|n-b|>(\log N)^{4s}
$$
for all $b\in B$. In the opposite case, $n$ will be said to be {\em bad (for N)}.

We consider the random variable (recall $\II$ is the indicator function of an event)
$$
X_N=\sum_{\substack{N/2\le n\le N\\ n\text{ is good}}} \II (F_n).
$$
We use the notation
$$
\mu_N=\E(X_N)\quad \text{ and }\quad \sigma_N^2=\Var (X_N).
$$
Our strategy is to prove that
\begin{equation}
\label{mmu}
\lim_{i \to + \infty} \mu_{N_i}= + \infty
\end{equation}
and that
\begin{equation}
\label{var}
\sigma_{N_i}^2\ll \frac{\mu_{N_i}^2}{\log N_i}.
\end{equation}
Then, using Chebychev's inequality in the form \eqref{cheby}, we get
\begin{equation}
\label{che}
\Prob\left(X_{N_i}<\frac{\mu_{N_i}}{2}\right)<\frac{4\sigma_{N_i}^2}{\mu_{N_i}^2}\ll \frac 1{\log N_i}.
\end{equation}
Now, Theorem \ref{th2bis} follows immediately from \eqref{che} and \eqref{mmu}.

From now on, we let $N$ be a term of the sequence $\left(N_i\right)_{i\in\mathbb{N}}$.

\subsection{Estimate of $\mu_N$}
\begin{prop}\label{mmmu}
We have $$\mu_N\ge N^{(1-c\lambda_s)(1+o(1))}.$$
\end{prop}
\begin{proof}
We have
\begin{eqnarray}\label{F}
\mu_N  
&=&\sum_{\substack{N/2\le n\le N\\ n\text{ good}}}\Prob(F_n).
\end{eqnarray}

Let $n$ be a good integer for $N$. Using Janson's first inequality (Theorem \ref{des1} (i)) we observe that
$$
\Prob (F_n)\ge \prod_{\omega\in \Omega_n}\left(1-\Prob (E_{\omega}) \right ).
$$
Using that $\log (1-x)=-x+O(x^2)$ we have
\begin{equation}\label{log}
\log(\Prob (F_n))\ge -\sum_{\omega\in \Omega_n}\Prob (E_{\omega})+O\left (\sum_{\omega\in \Omega_n}\Prob (E_{\omega})^2 \right ).
\end{equation}
On the one hand, since $n$ is good, we compute
\begin{eqnarray}\label{PEomega}
\sum_{\omega\in \Omega_n}\Prob (E_{\omega})&=& \frac 1{s^s} \sum_{\substack{b<n-(\log N)^{4s}\\b\in B}} \hspace{.3cm}
									\sum_{\substack{1 \leq x_s< \cdots <x_1\\x_1+\cdots+x_s=n-b}}(x_1\cdots x_s)^{-1+1/s}\\
	&=& \frac 1{s^s} \sum_{\substack{b<n-(\log N)^{4s} \\ b\in B}} 
		s^s\lambda_s(1+o(1))\nonumber \\ &\le & \lambda_s(1+o(1))B(N)\nonumber \\ &\le & c\lambda_s (1+o(1))\log N.	\nonumber	
\end{eqnarray}

On the other hand,
\begin{eqnarray*}
\sum_{\omega\in \Omega_n}\Prob (E_{\omega})^2& = &\sum_{\substack{b<n-(\log N)^{4s}\\ b\in B}}  \hspace{.3cm}
			\left( \frac 1{s^s} \right)^2	 \sum_{\substack{1 \leq x_s< \cdots <x_1\\ x_1+\cdots+x_s=n-b}}(x_1\cdots x_s)^{-2+2/s}\\
	&\ll & \sum_{\substack{b<n-(\log N)^{4s}\\ b\in B}}(n-b)^{-2+2/s} \sum_{\substack{1 \leq x_s< \cdots <x_2\\ x_2+\cdots+x_s\leq n-b}}(x_2\cdots x_s)^{-2+2/s}
\end{eqnarray*}
by noticing that $x_1 \geq (n-b)/s$ in each term of the internal sum. We further compute, $n$ being good,
\begin{eqnarray*}	
\sum_{\omega\in \Omega_n}\Prob (E_{\omega})^2	&\ll & \sum_{\substack{b<n-(\log N)^{4s}\\ b\in B}}(n-b)^{-2+2/s}\left (\sum_{x=1}^{n-b}x^{-2+2/s}\right )^{s-1}\\
	&\ll &\sum_{\substack{b<n\\ b\in B}}(\log^{4s} N)^{-2+2/s}\left (\sum_{x=1}^{n-b}x^{-1}\right )^{s-1}\\
	&\ll &\sum_{\substack{b<n\\ b\in B}}(\log N)^{-8s+8} (\log N)^{s-1}\\
	&\ll &(\log N)^{-7s+7} B(N)\\
	&\ll &(\log N)^{-6}.
\end{eqnarray*}
Thus, \eqref{log} and \eqref{PEomega} imply that
\begin{equation}
\label{Fn}
\Prob (F_n)\ge N^{-c\lambda_s(1+o(1))}
\end{equation}
when $n$ is good.

One computes that
\begin{eqnarray*}
|\{N/2 \le n\le N:\  n \text{ bad} \}|&  = & |\{N/2\le n\le N :\ |n-b|<(\log N)^{4s} \text{ for some } b\in B\}|\\
					& \le &\sum_{b<N} 2(\log N)^{4s}\\
					& \ll & (\log N)^{4s+1}.
\end{eqnarray*}

Thus, using equations \eqref{F}, \eqref{Fn} and this, we obtain
\begin{eqnarray*}
\mu_N&=&\sum_{\substack{N/2 \le n\le N\\ n\text{ good}}}N^{-c\lambda_s(1+o(1))}\\
 &\ge&  \sum_{\substack{N/2 \le n\le N}} N^{-c\lambda_s(1+o(1))}-  \sum_{\substack{N/2 \le n\le N\\ n\text{ bad}}} N^{-c\lambda_s(1+o(1))}\\
 & \ge &   N^{ (1+o(1))(1-c\lambda_s) } -O\left ((\log N)^{4s+1}\right )\\
 & \ge &   N^{ (1+o(1))(1-c\lambda_s) }
\end{eqnarray*}
since $1-c\lambda_s>0$.
\end{proof}

\subsection{Estimate of $\sigma_N^2$}
Let us recall now that, given a set $B$, its {\em difference set} $B-B$ is defined by
$$
B-B = \{ b-b' \text{ with } b,b' \in B \}.
$$

 \begin{lem}
 \label{BB}
 Let $B_N=\{b\le N \text{ with }\ b\in B\}$. Let $n<m\le N$ be two positive integers such that $m-n\not \in B_{N}-B_{N}$ then
 $$
 \Prob (F_n\cap F_m)\le \Prob (F_n) \Prob (F_m)
 \exp\left (2\sum_{\substack{\omega,\omega'\in \Omega_n\cup \Omega_m\\ \omega\sim \omega'}}
 \Prob (E_{\omega}\cap E_{\omega'})\right ).
 $$
 \end{lem}

 \begin{proof}
 We observe that
 $$
 F_n\cap F_m=\bigcap_{\omega\in \Omega_n\cup \Omega_m}E_{\omega}^c
 $$
 and that the condition $m-n\not \in B_N-B_N$ implies that $\Omega_n\cap \Omega_m=\emptyset$.  Janson's second inequality
 (Theorem \ref{des1} (ii)) applied to $\Omega =\Omega_n\cup \Omega_m$) implies that
 \begin{eqnarray*}
 \Prob(F_n\cap F_m)&\le &
 \prod_{\omega\in \Omega_n\cup \Omega_m}\Prob (E_{\omega}^c)
 \exp\left (2\sum_{{\substack{\omega,\omega'\in \Omega_n\cup \Omega_m\\ \omega\sim \omega'}}}\Prob(E_{\omega}\cap E_{\omega'})\right )\\
 & = & \prod_{\omega\in \Omega_n}\Prob (E_{\omega}^c)
 \prod_{\omega\in \Omega_m}\Prob(E_{\omega}^c)
 \exp\left (2\sum_{\substack{\omega,\omega'\in \Omega_n\cup \Omega_m\\ \omega\sim \omega'}}\Prob (E_{\omega}\cap E_{\omega'})\right )\\
 &\le & \Prob (F_n) \Prob (F_m)   \exp\left (2\sum_{ \substack{\omega,\omega'\in \Omega_n\cup \Omega_m\\ \omega\sim \omega'} }\Prob(E_{\omega}\cap E_{\omega'})\right )
 \end{eqnarray*}
 using Janson's first inequality (Theorem \ref{des1} (i)) applied to $\Omega_n$ and to $\Omega_m$.
 The lemma is proved.
 \end{proof}

\begin{lem}
\label{ij}
Let $N,n,m$ be integers. 
If $n$ and $m$  are good for $N$, then
$$
\sum_{\substack{\omega\in \Omega_n,\omega'\in \Omega_m\\ \omega\sim \omega'}}
\Prob(E_{\omega}\cap E_{\omega'})\ll \frac{1}{\log N}.
$$
\end{lem}

\begin{proof}
We can write
$$
\sum_{\substack{\omega\in \Omega_n,\omega'\in \Omega_m\\ \omega\sim \omega'}}
\Prob(E_{\omega}\cap E_{\omega'})=\sum_{\substack{1 \leq b<n\\1 \leq b'<m\\b,b' \in B}}\hspace{.2cm}\sum_{k=1}^{s-1}\Delta_{n,m}(k;b,b')
$$
where, for $k \geq 1$,
$$
\Delta_{n,m}(k;b,b')=\sum_{\substack{\omega\in \Omega_n,\omega'\in \Omega_m\\
\sigma(\omega)=n-b,\ \sigma(\omega')=m-b'\\|\omega\cap \omega'|=k}}P(E_{\omega}\cap E_{\omega'}).$$
Assume that $n-b\le m-b'$. Thus,
\begin{eqnarray*}
\Delta_{n,m}(k;b,b')& \ll & \hspace{-0.5cm}\sum_{\substack{1 \leq x_1,\dots,x_k\\ x_1+\dots +x_k<n-b}} (x_1\cdots x_k)^{-1+1/s}
\left (\sum_{\substack{1 \leq x_{k+1},\dots,x_s\\x_{k+1}+\cdots +x_s=n-b-(x_1+\cdots+x_k)}}(x_{k+1}\cdots x_s)^{-1+1/s}\right )\\
& &\hspace{3.5cm}\times \left ( \sum_{\substack{1 \leq y_{k+1},\dots,y_s\\y_{k+1}+\cdots +y_s=m-b'-(x_1+\cdots+x_k)}}(y_{k+1}\cdots y_s)^{-1+1/s}\right).
\end{eqnarray*}
Lemma \ref{t1} (i) applied twice
shows that

\begin{eqnarray*}
\Delta_{n,m}(k;b,b')&\ll &\hspace{-0.5cm}\sum_{\substack{1 \leq x_1,\dots,x_k\\x_1+\dots +x_k<n-b}}\hspace{-0.5cm}(x_1\cdots x_k)^{-1+\frac 1s}(n-b-(x_1+\cdots+x_k)   )^{-\frac ks}(m-b'-(x_1+\cdots+x_k)   )^{-\frac ks}\\
&\ll &\sum_{\substack{1 \leq x_1,\dots,x_k\\x_1+\dots +x_k<n-b}}\hspace{-0.5cm}(x_1\cdots x_k)^{-1+1/s}(n-b-(x_1+\cdots+x_k)   )^{-2k/s}\\ &\ll& (n-b)^{-1/s}\log (n-b)\\
&\ll &\frac{1}{\log^3 N}.
\end{eqnarray*}
since $(\log N)^{4s} \le n-b \le N$.

If  $m-b'<n-b$ we proceed in the same way.
 Thus,
\begin{eqnarray*}
\sum_{\substack{\omega\in \Omega_n,\omega'\in \Omega_m\\ \omega\sim \omega'}}\Prob (E_{\omega}\cap E_{\omega'})&\ll&
\sum_{\substack{1 \leq b<n\\ b \in B}}\sum_{\substack{1 \leq b'<m\\ b' \in B}} \frac{1}{\log^3 N}\\
&\ll &\frac{(B(N))^2}{\log^3 N}\\
&\ll & \frac{1}{\log N},
\end{eqnarray*}

hence the result.
\end{proof}

\begin{cor}
\label{c1}
Let $N,n,m$ be integers. If $n$ and $m$ 
are good for $N$ and $m-n\not\in B_N-B_N$ then
$$
\Prob (F_n\cap F_m)-\Prob (F_n)\Prob (F_m)\ll \frac{1}{\log N} \Prob (F_n)\Prob (F_m).
$$
\end{cor}

\begin{proof}
Lemma \ref{BB} implies that
$$
\Prob(F_n\cap F_m)-\Prob (F_n)\Prob (F_m)\le
\Prob (F_n) \Prob (F_m) \left (\exp\left (2\sum_{\substack{\omega,\omega'\in \Omega_n\cup \Omega_m\\ \omega \sim \omega'}}
\Prob (E_{\omega}\cap E_{\omega'})\right )-1\right ).
$$
We observe that
\begin{eqnarray*}
\sum_{\substack{\omega,\omega'\in \Omega_n\cup \Omega_m\\ \omega \sim \omega'}}
\Prob (E_{\omega}\cap E_{\omega'})&=&\sum_{\substack{\omega,\omega'\in \Omega_n\\ \omega \sim \omega'}}\Prob (E_{\omega}\cap E_{\omega'})
								+\sum_{\substack{\omega,\omega'\in \Omega_m\\ \omega \sim \omega'}}\Prob (E_{\omega}\cap E_{\omega'})\\
&&  \hspace{3cm}+  \sum_{\substack{\omega\in \Omega_n,\omega'\in \Omega_m\\ \omega \sim \omega'}}\Prob (E_{\omega}\cap E_{\omega'}).
\end{eqnarray*}
We finish the proof applying Lemma \ref{ij} to the three sums (with $n=m$ or not) and using  the estimate $e^x-1\sim x$ when $x$ approaches $0$.
\end{proof}

\begin{prop}
\label{lemfinal}
The following estimate holds
$$\sigma_N^2\ll \frac{\mu_N^2}{\log N}.$$
\end{prop}

\begin{proof}
A standard calculation shows that
$$\sigma_N^2=2\sum_{\substack{N/2 \leq n <m\le N\\ n,m\text{ good}}} \Big( \Prob (F_n\cap F_m)-\Prob (F_n)\Prob (F_m) \Big)+\sum_{\substack{N/2 \leq n \le N\\ n\text{ good}}} \Big( \Prob (F_n)-\Prob^2 (F_n) \Big).$$
We decompose
$$\sigma_N^2=2\Sigma_1+2\Sigma_2+\Sigma_3,$$
where
\begin{eqnarray*}
\Sigma_1&= &\sum_{\substack{N/2<n <m\le N\\ n-m\not \in B_N-B_N\\ n,m\text{ good}}} \Big(\Prob (F_n\cap F_m)-\Prob (F_n)\Prob (F_m )\Big),\\
\Sigma_2&=&\sum_{\substack{N/2 \leq n< m\le N\\ n-m\in B_N-B_N\\ n,m\text{ good}}} \Big( \Prob (F_n\cap F_m)-\Prob (F_n)\Prob (F_m)\Big),\\
\Sigma_3&=&  \sum_{\substack{N/2 \leq n \le N\\ n\text{ good}}} \Big( \Prob (F_n)-\Prob^2 (F_n) \Big).
\end{eqnarray*}

It is clear that $$\Sigma_3\le \mu_N.$$

To bound $\Sigma_2$  from above, we use the trivial upper bound
$$
\Prob (F_n\cap F_m)-\Prob (F_n)\Prob(F_m)\le \Prob(F_m)
$$
and get,
 for $\Sigma_2$,
\begin{eqnarray*}
\Sigma_2  & \le & \sum_{\substack{N/2\leq m<N\\  m \text{ good}}}\Prob (F_m)|\{N/2\leq n \leq N \text{ such that } n\in B_N-B_N+m\}| \\
		& \le & |B_N - B_N| \sum_{\substack{N/2\leq m<N\\  m \text{ good}}}\Prob (F_m)\\
		& \ll & |B_N |^2 \sum_{\substack{N/2\leq m<N\\  m \text{ good}}}\Prob (F_m)\\
		& \ll & \log^2 N\ \mu_N.
\end{eqnarray*}
Finally, by Corollary \ref{c1}, we have
$$
\Sigma_1\ll \frac{1}{\log N}\sum_{\substack{N/2<n <m\le N\\ n-m\not \in B_N-B_N\\ n,m\text{ good}}}\Prob (F_n)\Prob (F_m)
\leq \frac{1}{\log N} \Big( \sum_{\substack{N/2 \leq n \le N\\ n \text{ good}}} \Prob (F_n) \Big)^2
= \frac{\mu_N^2}{\log N}.
$$
Adding the three contributions $\Sigma_1,\Sigma_2$ and $\Sigma_3$ we have
\begin{equation}
\label{sisi}
\sigma_N^2\ll \frac{\mu_N^2}{\log N}+\log^2 N\ \mu_N+\mu_N\ll \mu_N^2\left(\frac 1{\log N}+\frac{\log^2N}{\mu_N}\right ).
\end{equation}

We let
$$
\varepsilon=\frac{1-c\lambda_s}2 >0
$$
and notice that Proposition \ref{mmmu} implies
\begin{equation}
\label{Znn}
\mu_N \geq N^{2 \varepsilon + o(1)}\gg \log^3 N.
\end{equation}
We obtain the Proposition after plugging \eqref{Znn} in the last term of \eqref{sisi}.
\end{proof}

\section{The limit case of Theorems \ref{th2} and \ref{th2bis}:\\ Sequences at the threshold}
\label{sec6}

Theorems  \ref{th2} and  \ref{th2bis} being proved, it is natural to wonder what happens for sequences $B$ at the
threshold, namely satisfying
$$
\liminf_{n\to \infty}\frac{B(n)}{\log n}=\lambda_s^{-1}.
$$
In this paragraph, we show how to build sequences at the threshold satisfying either the conclusion of Theorem \ref{th2}
or of Theorem \ref{th2bis}.

Indeed, consider for example the sequence $B$ defined by the counting function
$$
B(n)=\left \lfloor \lambda_s^{-1}\log n+2\lambda_s^{-1}\log \log n\right \rfloor.
$$
We can mimic the proof of Theorem \ref{th2} (although we have to change $m=n/2$ now).

We'll use the following refinement of Lemma 1, (iii)
\begin{equation}\label{straightforward}
\sum_{\substack{1\le x_s<\cdots <x_1\\ x_1+\cdots +x_s=n}}(x_1\cdots x_s)^{-1+1/s} = s^s\lambda_s+O(n^{-1/(s+1)}).
\end{equation}
Hint: we let $g(n) = n^{1/(s+1)}$, break the sum over $x_s$ at $g(n)$. In the sum with $x_s \ge g(n)$ we recognize (up to the right gamma factor) a Riemann sum for the integral $\idotsint (t_1\ldots t_s)^{-1+1/s} dt_1\ldots dt_s$ over the part of the hyperplane $t_1+\cdots +t_s=1$ limited by $ g(n)/n <t_s<\cdots <t_1 \le 1$; the error in the approximation of the integral by the Riemann sum is $O(1/g(n))$; the error in the truncation of the sum (cf. the proof of part (iii) of Lemma 1) is $O((g(n)/n)^{1/s})$ and so is the error in the truncation of the integral. The resulting global error is $O(n^{1/(s+1)})$, which is enough for our purpose. By looking carefully at what occurs around 0 and integrating the error in the approximation, one can reduce the error term to $O(n^{-1/s})$.

Equation (\ref{straightforward}) leads to
\begin{eqnarray*}\sum_{\omega\in \Omega_n}\Prob(E_{\omega})&=&\frac 1{s^s}\sum_{b<n/2}\sum_{\substack{1\le x_s<\cdots <x_1\\ x_1+\cdots +x_s=n}}(x_1\cdots x_s)^{-1+1/s}\\ &=&B(n/2)\left (\lambda_s +O(n^{-1/(s+1)})\right )\\
&=& \log n+2\log \log n+O(1).\end{eqnarray*}
Following the same reasoning as in Theorem \ref{th2}  we get $$\Prob(F_n)\le e^{-(\log n+2\log \log n+O(1))  }\ll \frac 1{n\log^2n}.$$
Thus, $\sum_n \Prob (F_n)<\infty$ and we can apply the Borel Cantelli Lemma to conclude that the sequence $B$ is almost surely  complementary sequence of a pseudo $s$-th power.

Conversely, consider for example a sequence $B$ defined by the counting function
$$
B(n)=\left \lfloor \lambda_s^{-1}\log n-t(n)\right \rfloor,
$$
where $t(n)$ is an increasing function with $t(n)=o(\log n)$.
We can mimic the proof of Theorem \ref{th2bis} with the only difference  that now the exponent
$2\epsilon+o(1)$ in \eqref{Znn} is $2\epsilon_N\sim \lambda_st(N)/\log N$. So, we can take for $t(n)$
any function such that $\mu_N\gg N^{\epsilon_N}\gg \log^3 N$. For example, the choice
$$
t(N)=4\lambda_s^{-1}\frac{\log \log N}{\log N}
$$
is satisfactory.

\end{document}